\documentclass[10pt]{article}
\usepackage{amsfonts}
\usepackage{amsthm}
\usepackage{amsmath}
\usepackage{amssymb}
\usepackage{tikz,tikz-cd}
\usetikzlibrary{matrix,arrows,decorations.pathmorphing}
\usepackage{thmtools}
\usepackage{thm-restate}
\usepackage{hyperref}
\usepackage{cleveref}
\usepackage{cite}

\newcommand{\ds}{\displaystyle}
\newcommand{\normal}{\trianglelefteq}
\newcommand{\F}{\mathbb{F}}
\newcommand{\Z}{\mathbb{Z}}
\newcommand{\Q}{\mathbb{Q}}
\newcommand{\C}{\mathbb{C}}

\newcommand{\Gal}{\textnormal{Gal}}
\newcommand{\Hom}{\textnormal{Hom}}
\newcommand{\Cl}{\textnormal{Cl}}
\newcommand{\im}{\textnormal{im }}
\newcommand{\Frob}{\textnormal{Frob}}

\newcommand{\rk}{\textnormal{rk}}
\newcommand{\Aut}{\textnormal{Aut}}

\newcommand{\Res}{\textnormal{Res}}

\newcommand{\disc}{\textnormal{disc}}
\newcommand{\Leg}[2]{\left(\frac{#1}{#2}\right)}

\usepackage[margin=1.5in]{geometry}

\newtheorem*{lemma*}{section}

\newtheorem{theorem}{Theorem}[section]
\newtheorem{lemma}[theorem]{Lemma}
\newtheorem{definition}{Definition}[section]
\newtheorem{corollary}[theorem]{Corollary}

\newtheorem{proposition}[theorem]{Proposition}

\theoremstyle{definition}
\newtheorem{example}[theorem]{Example}

\begin{document}

\title{Classifying Unramified Metabelian Extensions Using Lemmermeyer Factorizations}

\author{Brandon Alberts}

\maketitle

\abstract{We study solutions to the Brauer embedding problem with restricted ramification. Suppose $G$ and $A$ are a abelian groups, $E$ is a central extension of $G$ by $A$, and $f:\Gal(\overline{\Q}/\Q)\rightarrow G$ a continuous homomorphism. We determine conditions on the discriminant of $f$ that are equivalent to the existence of an unramified lift $\widetilde{f}$, i.e. a continuous homomorphism making the following diagram commute:
\[
\begin{tikzcd}
{} & & & \Gal(\overline{\Q}/\Q)\dlar[dashed,swap]{\widetilde{f}}\dar{f}\\
1\rar & A \rar & E \rar & G \rar & 1
\end{tikzcd}
\]
As a consequence of this result, we use conditions on the discriminant of $K$ for $K/\Q$ abelian to classify and count unramified nonabelian extensions $L/K$ normal over $\Q$ where the (nontrivial) commutator subgroup of $\Gal(L/\Q)$ is contained in its center. This generalizes a result due to Lemmermeyer, which states that a quadratic field $\Q(\sqrt{d})$ has an unramified extension normal over $\Q$ with Galois group $H_8$ the quaternion group if and only if the discriminant factors $d=d_1 d_2 d_3$ as a product of three coprime discriminants, at most one of which is negative, satisfying the following condition on Legendre symbols:
\[
\Leg{d_i d_j}{p_k}=1
\]
for $\{i,j,k\}=\{1,2,3\}$ and $p_i$ any prime dividing $d_i$.
}

\section{Introduction}

The question of existence of unramified abelian extensions of number fields is controlled by class field theory, with the result that the maximal unramified abelian extension of a number field $K$ has Galois group isomorphic to $Cl(K)$, the class group of $K$. Class groups are computable in individual instances and Cohen-Lenstra heuristics \cite{cohen-lenstra1} with related modifications \cite{gerth1} give conjectural evidence for the distribution of isomorphism classes of class groups over certain families of number fields. Yet overall there are still many open questions with regards to how often number fields have a particular class group or the relationships between class groups in families of number fields.

The most basic results for computing part of the class group in families of number fields come from genus theory. If $k$ is a quadratic field of discriminant $d$, then the $2$-rank $\rk_2 Cl(k) = \omega(d)-1$, where $\omega(d)=$ the number of prime divisors of $d$. This gives precisely the number of unramified quadratic extensions of a given quadratic number field. Similarly, unramified extensions of abelian number fields $L/k$ can be found using Kronecker-Weber whenever we require $L/\Q$ to be abelian.

One of the earliest results that go beyond Kronecker-Weber is due to Fueter \cite{fueter1}, R\'edei and Reichardt \cite{redei1}\cite{redei-reichardt1}\cite{reichardt1}. Their works prove the following classification:

\begin{theorem}\label{theorem:C4}
Let $k$ be a quadratic number field with discriminant $d$. Then there exists an unramifed $C_4$-extension $K/k$ if and only if there exists a factorization $d=d_1d_2$ into coprime discriminants such that
\[
\Leg{d_1}{p_2}=\Leg{d_2}{p_1}=1
\]
for all primes $p_j\mid d_j$. Moreover, there exist $2^{\omega(d)-2}$ unramified $C_4$-extensions $K/k$ satisfying $\Q(\sqrt{d_1},\sqrt{d_2})\subset K$ for each such nontrivial factorization.
\end{theorem}

There are other results of a similar flavor to this one classifying unramified extensions of families of number fields. For instance, there are results for the $\ell$- and $\ell^2$- torsion of the class group over cyclic degree $\ell$ fields \cite{gras1}\cite{inaba1}\cite{martinet1}\cite{nakagoshi1}\cite{odai1} and some results for nonabelian $p$-extensions of quadratic fields or cyclic fields of prime degree $q\ne p$ due to Nomura \cite{nomura1}\cite{nomura2}\cite{nomura3}\cite{nomura4}. Closest to the direction of this paper are results due to Lemmermeyer classifying unramified $G$-extensions of quadratic fields for $G$ one of several small nonabelian 2-groups \cite{lemmermeyer1}\cite{lemmermeyer2}, for example:

\begin{theorem}\label{theorem:H8}
Let $k$ be a quadratic number field of discriminant $d$ and $H_8$ the quaternion group of order $8$. Then there exists an unramified $H_8$-extension $K/k$ normal over $\Q$ if and only if there exists a factorization $d=d_1d_2d_3$ into three coprime discriminants, at most one of which is negative, such that
\[
\Leg{d_1d_2}{p_3}=\Leg{d_1d_3}{p_2}=\Leg{d_2d_3}{p_1}=1
\]
for all primes $p_j\mid d_j$. Moreover, there exist $2^{\omega(d)-3}$ unramified $H_8$-extensions $K/k$ normal over $\Q$ satisfying $\Q(\sqrt{d_1},\sqrt{d_2},\sqrt{d_3})\subset K$ for each such nontrivial factorization.
\end{theorem}

All of these results follow the same theme: existence of an unramified extension corresponds to a vanishing condition on Dirichlet characters. As a consequence these results have been used to study the arithmetic statistics of unramified extensions in \cite{alberts1},\cite{alberts-klys1},\cite{fouvry-kluners1}, and \cite{klys1}.

The purpose of this paper will be to generalize these results using class field theory and the theory of embedding problems. This approach follows the same basic idea as Nomura \cite{nomura1}\cite{nomura2}\cite{nomura3}\cite{nomura4}, by using local conditions to say something about the solution to an embedding problem.

Suppose we have abelian groups $G$ and $A$, and $E$ a central extension of $G$ by $A$ whose $2$-coclass in $H^2(G,A)$ is represented by $[E]$. The Brauer embedding problem for a continuous homomorphism $f:G_\Q:=\Gal(\overline{Q}/\Q)\rightarrow G$ and an extension $E$ of $G$ is the question of whether or not there exits a continuous homomorphism $\widetilde{f}:G_\Q\rightarrow E$ such that the following diagram commutes:

\[
\begin{tikzcd}
{} & G_\Q\dar{f}\dlar[swap]{\widetilde{f}}\\
E \rar & G \rar &1
\end{tikzcd}
\]

We know that the a solution to the Brauer embedding problem for central extensions exists if and only if the solutions to the corresponding local embedding problems exist \cite{serre1}. Our main theorem will give explicit conditions for when we can find an unramified (at finite places) solution ot the Brauer embedding problem, namely solutions where $\ker f\cap I_p\le \ker \widetilde{f}$ with $I_p\le G_\Q$ the inertia group at $p$ defined up to conjugation. This will provide the necessary framework to classify certain unramified metabelian extensions of abelian number fields in a way directly generalizing the classification due to Lemmermeyer \cite{lemmermeyer1}\cite{lemmermeyer2}.

Before stating the main result, we introduce the following notation: for a prime $p>2$ fix a representative $\tau_p$ for the generator of inertia $I_p\le G_{\Q}^{ab}$ (for $p=2$ we choose generators $\tau_{2,0}$ and $\tau_{2,1}$ such that $[\tau_{2,0},\tau_{2,1}]=1$ and $(\tau_{2,0}\tau_{2,1})^2=1$). We also have a bilinear map $[,]_E:G\times G\rightarrow A$ given by the comutator of the lift to $E\times E$ and use $\exp(A)$ to denote the exponent of the group $A$, i.e. the smallest $n\in \Z_{>0}$ such that $nA=0$. Let $\Res^{G}_{H}:H^2(G,A)\rightarrow H^2(H,A)$ denote the restriction map on cohomology. Then the main result is as follows:

\begin{restatable}{theorem}{main}
\label{thm:main}
Given $f:G_\Q\rightarrow G$ let $Y_E=\{y\in G: \Res^G_{f(I_p)}([E])=0\}$ and $\disc(f)=\prod p^{e_p}$. Define $d_y = \prod_{f(\tau_p)=y} p^{e_p/[f(G_\Q):\langle y\rangle]}$. Then there exists an unramified (everywhere except possibly at the infinite place) solution to the Brauer embedding problem for $[E]$ if and only if
\begin{enumerate}
\item{
$\disc(f) = \prod_{y\in Y_E} d_y^{[f(G_\Q):\langle y\rangle]}$
}
\item{
For every prime $p\mid d$
\begin{align*}
\sum_{y\in Y_E}\Leg{p}{d_y}_{\exp(A)}[y,f(\tau_p)]_E= 0
\end{align*}
where $\Leg{p}{d_y}_{\exp(A)}\in\Z/\exp(A)\Z$.
}
\end{enumerate}
\end{restatable}

Here $\Leg{\cdot}{d_y}_{\exp(A)}$ is a Dirichlet character, which we will describe in greater detail in the body of the paper.

The paper proceeds as follows: In section \ref{embed} we begin with a review of the Brauer embedding problem with restricted ramification, which will then be converted into group theoretic conditions. In section \ref{factor}, we present the proof of the main theorem. Section \ref{meta} describes how our result can be used to classify unramified metabelian extensions $L/K$ of abelian number fields whose Galois group over $\Q$ has its commutator subgroup contained in its center, generalizing the results of Fuerter, R\'edei, Reichardt, and Lemmermeyer. In section \ref{HeisenbergGroups}, we walk through an example classifying all unramified $H(\ell^3)$-extensions of cyclic degree $\ell$ fields tamely ramified at exactly three primes, where $\ell$ is an odd prime and $H(\ell^3)$ is the nonabelian group of order $\ell^3$ and exponent $\ell$ (called the Heisenberg group).

\section{The Unramified Brauer Embedding Problem}\label{embed}

Let $G_\Q=\Gal(\overline{\Q}/\Q)$ denote the absolute Galois group of $\Q$ with $D_p$ and $I_p$ the corresponding decomposition and inertia groups at the prime $p$ (defined as subgroups of $G_\Q$ up to conjugacy for $p$ finite or infinite). Given an extension $E$ of a group $G$ by an abelian group $A$, suppose we have a commutative diagram as follows:
\[
\begin{tikzcd}
{}&{}&{}&G_\Q\dar{f}\dlar[dashed][swap]{\widetilde{f}}\\
1\rar & A\rar & E\rar{\pi} & G\rar & 1
\end{tikzcd}
\]
Then the Brauer embedding problem asks when $f$ lifts to a continuous homomorphism $\widetilde{f}:G_\Q\rightarrow E$ that makes the diagram commute. This problem is ``solved" in a certain sense in the case of $A$ abelian and $E$ a central extension, in that a solution exists if and only if a solution exists to the corresponding local embedding problems:
\[
\begin{tikzcd}
{}&{}&{}&D_p\dar{f_p}\dlar[dashed][swap]{\widetilde{f}_p}\\
1\rar & A\rar & E\rar{\pi} & G\rar & 1
\end{tikzcd}
\]
where $f_p=f|_{D_p}$ \cite{malle-matzat1}\cite{serre1}.

We call such a lift $\widetilde{f}$ unramified over $f$ if $\ker f\cap I_p \le \ker \widetilde{f}$. Let $K=\overline{\Q}^{\ker f}$, $L=\overline{\Q}^{\ker\widetilde{f}}$, and $K_p$, $L_p$ the completions at a prime above $p$. Then $\widetilde{f}/f$ is unramified if and only if $L/K$ is unramified, i.e. if and only if $\widetilde{f}$ factors through $\Gal(K^{ur,ab}/K)$ where $K^{ur,ab}$ is the maximal unramified abelian extension of $K$, otherwise known as its Hilbert class field.

The solutions to the Brauer embedding problem are a priori independent of ramification. With a little extra work, we can determine when a solution can be found with specific ramification. Some results of this nature in more specific instances were proven in \cite{nomura1}\cite{nomura2}\cite{nomura3} by Nomura and some general results can be found in \cite{serre1}.

\begin{theorem}
Given $f:G_\Q\rightarrow G$ a continuous homomorphism with $f(D_p)$ abelian, there exists a solution to the Brauer embedding problem $\widetilde{f}:G_\Q\rightarrow E$ unramified at finite places if and only if $\Res^G_{f(D_p)}([E])$ is abelian and $\Res^G_{f(I_p)}([E])=0$ for all primes (finite and infinite) $p$.
\end{theorem}

\begin{proof}
For the forward direction, we have that $f(I_p)\cong I_p/(\ker f\cap I_p)=I_p/(\ker \widetilde{f}\cap I_p)\cong \widetilde{f}(I_p)$. This implies that there is a section $\tau_p:I_p\rightarrow \pi^{-1}(I_p)$, forcing $\Res^G_{f(I_p)}([E])=0$.

Additionally, we have $L_p/K_p$ is unramified, with $\widetilde{f}(D_p)=\Gal(L_p/\Q)$ and $f(D_p)=\Gal(K_p/\Q)$ abelian. Any unramified extension of an abelian local field is also abelian by a standard result, so $\widetilde{f}(D_p)$ is abelian. Then $\Res^G_{f(D_p)}([E])$ has an abelian subgroup which surjects onto $f(D_p)$. Since this is a central extension of $f(D_p)$ this implies that $\Res^G_{f(D_p)}([E])$ must also be abelian.

For the converse, note that $\Res^G_{f(I_p)}([E])=0$ implies the local embedding problem is solvable at all places, so there exists a possibly ramified lift $\widetilde{f}$ with $\pi\widetilde{f}=f$. $\Res^G_{f(I_p)}([E])=0$ implies that $\widetilde{f}(I_p)$ is a trivial extension of $f(I_p)$. $\Res^G_{f(D_p)}([E])$ abelian implies that inertia factors through procyclic $\Z_p^{\times}$ with finite image for both $f$ and $\widetilde{f}$. Therefore we must have $\widetilde{f}(I_p)\cong C_{n_p}\times f(I_p)$ with $(n_p,|f(I_p)|)=1$. Let $\rho_p:I_p\rightarrow A$ be $\widetilde{f}|_{I_p}$ composed with projection onto the first coordinate $C_{n_p}$ and embedded into $A$. Define the map
\begin{align*}
\rho=(\rho_p):G_\Q=\prod I_p^{ab} \rightarrow A
\end{align*}
Then we have a map
\begin{align*}
g=f\times \rho: G_\Q\rightarrow G\times A
\end{align*}
which satisfies $g(I_p)\cong\widetilde{f}(I_p)$ since $(|f(I_p)|,|\rho(I_p)|)=1$ by $|\rho(I_p)|=n_p$. Then we have a map $\widetilde{f}\times \rho:G_\Q\rightarrow E\times A$ which composed with $\pi$ gives $g$, and composed with projection onto the first coordinate gives $\widetilde{f}$. We have a subgroup $A\times A \normal E\times A$ and
\begin{align*}
(A\times A)\cap (\widetilde{f} \times \rho)(I_p)= (\rho_p\times \rho_p)(I_p)
\end{align*}
is a diagonal embedding of $I_p$ into $A\times A$. This implies that $(A\times A)\cap (\widetilde{f}\times \rho)(I_p)\le \Delta$ the diagonal subgroup. Let $\pi_\Delta:E\times A\rightarrow E$ be the quotient map by $\Delta$ (in other words, we are adding the extensions $[E]$ and $[G\times A]$ together in $H^2(G,A)$, which gives back $[E]$). It then follows that $\pi_{\Delta}\circ(\widetilde{f}\times \rho):G_\Q \rightarrow E$ is a morphism with
\begin{align*}
\pi\pi_{\Delta}(\widetilde{f}(x),\rho(x)) &=\pi(\widetilde{f}(x)\rho(x)^{-1})\\
&=f(x)
\end{align*}
and
\begin{align*}
A\cap \pi_{\Delta}(\widetilde{f}\times \rho)(I_p)=0
\end{align*}
which implies that if $\ker f\cap I_p \le \ker \pi_{\Delta}(\widetilde{f}\times \rho)$ so that it is unramified at all finite places.
\end{proof}

We can also ask that $\widetilde{f}$ be unramified at the infinite place. This is trivially true if $f$ is ramified at $\infty$ or if $2\nmid |A|$. Moreover, if $2\mid f(I_p)$ for some finite prime $p\equiv 3\mod 4$, then we can take $g:G_\Q\rightarrow \Gal(\Q(\sqrt{-p})/\Q)$ and combine it with $\widetilde{f}$ ramified at infinity in the same way as in the proof of this theorem to get a new lift unramified at infinity. As for when $2\mid |f(I_p)|$ only for primes $p\equiv 1\mod 4$, the question becomes more difficult. Lemmermeyer addressed this when he studied unramified $H_8$-extensions of quadratic fields \cite{lemmermeyer1}, but when studying the expected number of such extensions Alberts and Klys avoid using this result by showing that it only contributes to the error term \cite{alberts-klys1}. Although we do not address ramification at $\infty$ in this paper, we expect the methods used by Lemmermeyer to extend to the results in our paper without too much difficulty.

Define $\Frob_p$ to be the lift of the Frobenius automorphism in $\Gal(\overline{F}_p/\F_p)$ to $G_\Q$, defined up to conjugacy and modulo $I_p$. The we have the following corollary:

\begin{corollary}
Given $f:G_\Q\rightarrow G$ with $f(D_p)$ abelian, and $[E]\in H^2(G,A)$ such that $\Res^G_{f(I_p)}([E])=0$, then there exists a lift $\widetilde{f}:G_\Q\rightarrow E$ unramified at finite places if and only if $[f(\Frob_p),f(I_p)]_E=0$ where $[,]_E:G\times G\rightarrow E$ is the commutator of $E$.
\end{corollary}

\begin{proof}
$E$ is a central extension of $G$, so the commutator factors through $G\times G$. Notice that $[a^g,b^g]_E=[a,b]_E^g$ respects conjugation, so WLOG we can fix a representative of $D_p$ and $I_p$ up to conjugation. Suppose $x,y\in D_p$ are both representatives of Frobenius $\Frob_p$. Because $f(D_p)\le G$ is abelian, we have $[f(x),f(I_p)]_E, [f(y),f(I_p)]_E\le \ker \pi=A$. Suppose $[f(x),f(I_p)]_E=0$. Then $[f(I_p),f(I_p)]_E=0$ since $\Res^G_{f(I_p)}([E])=0$ implies $\pi^{-1}(I_p)$ is a trivial extension of the abelian $f(I_p)$. Then for any $i\in \widetilde{f}(I_p)$
\begin{align*}
[\widetilde{f}(x),i] i\widetilde{f}(x) &=\widetilde{f}(x)i\\
&=\widetilde{f}(y)\widetilde{f}(y^{-1}x)i\\
&=\widetilde{f}(y)i\widetilde{f}(y^{-1},x)\\
&=[\widetilde{f}(y),i]i\widetilde{f}(y)\widetilde{f}(y^{-1}x)\\
&=[\widetilde{f}(y),i] i\widetilde{f}(x)
\end{align*}
noting that $\widetilde{f}(y^{-1}x)\in\widetilde{f}(I_p)$. Thus $[f(\Frob_p),f(I_p)]_E$ is a well-defined subgroup of $A$.

$\Res^G_{f(D_p)}([E])$ is abelian if and only if a representative of Frobenius commutes with inertia (i.e. $\widetilde{f}(D_p)$ abelian), and so the result follows from the previous theorem.
\end{proof}

\section{Lemmermeyer Factorization}\label{factor}

In this section we will translate the conditions in Corollary 1.1 to conditions on discriminants in the specific case that $G$ is abelian. Fix an extension $[E]\in H^2(G,A)$ and the quotient map $\pi:E\rightarrow G$. Then $[,]_E:G\times G\rightarrow A$ is a well-defined bilinear map for all elements, not just the image of Frobenius and inertia. For the remainder of this section, let unramified refer only to finite places.

The previous section tells us that a continuous homomorphism $f:G_\Q\rightarrow G$ has an unramified lift $\widetilde{f}:G_\Q\rightarrow E$ if and only if $\Res^G_{f(I_p)}([E])=0$ and $[f(\Frob_p),f(I_p)]_E=0$ for all primes $p$ (including the infinite prime).

For convenience, for $p$ odd we will use $\tau_p$ for the generator of $I_p^{ab}$, and WLOG consider $f:G_\Q^{ab}\rightarrow G$. (For the prime $2$, we have two generators $\tau_{2,0}$ and $\tau_{2,1}$ of order $2^\infty$, whose product $\tau_{2,0}\tau_{2,1}$ has order $2$. Here, $I_{\infty}\mapsto\langle\tau_{2,0}\tau_{2,1}\rangle$ is the ramification at infinity under the quotient map $G_\Q^{ab}\rightarrow\langle \tau_{2,0},\tau_{2,1}\rangle$). We also use $\disc(f)$ to denote the discriminant of $K=\overline{\Q}^{\ker f}$.

\begin{lemma}
\label{lemma:factorization} 
Let $Y_E=\{g\in G: \Res^G_{\langle g\rangle}([E])=0\}$. Then there is a one-to-one correspondence between maps $f:G_\Q^{ab}\rightarrow G$ satisfying $\Res^G_{f(I_p)}([E])=0$ with $\disc(f)=d=\prod p^{e_p}$ and factorizations $\disc(f)=\prod_{y\in Y_E} (d_y)^{[f(G_\Q):\langle y\rangle]}$ into coprime discriminants (except for at most two allowed to be even) where
\begin{enumerate}
\item{$|y|\mid (p-1)p^\infty$ for every prime dividing $d_y$ where $|y|$ is the order of $y$,}

\item{If $d_y,d_{y'}$ are even with $y\ne y'$ then $yy'$ has order dividing $2$. If $d_y$ is a unique even factor and $4\mid\mid d_y$, then $y$ has order dividing $2$,}

\item{$\ds\prod_{2\mid d_y<0}y\prod_{d_y<0}y^{|y|/2}\in Y_E$,}
\end{enumerate}
given by $d_y=\prod_{f(\tau_p)=y}(p^*)^{e_p/[f(G_\Q):\langle y\rangle]}$ where $p^*=(-1)^{\frac{p-1}{2}}$ for odd primes, and $2^*=-2$.
\end{lemma}

This lemma is a consequence of Kronecker-Weber. The proof is somewhat long and cumbersome, so we will first demonstrate the content of this theorem with an example from genus theory both to motivate the result and convince any reader not interested in trudging through the proof.

\begin{example}
Since
\[
G_\Q^{ab}/2G_\Q^{ab}=\langle \tau_{2,0},\tau_{2,1},\tau_p \text{ $p$ an odd prime}: \tau_{2,0}^2=\tau_{2,1}^2=\tau_p^2=1\rangle
\]
a $C_2^n$-extension of odd discriminant $d$ corresponds to a homomorphism
\[
f:\langle \tau_p, p\mid d:\tau_p^2=1\rangle\rightarrow C_2^n.
\]
We can produce a factorization from this map as follows: let $d_y=\prod_{p:f(\tau_p)=y}p^*$ where $p^*=(-1)^{\frac{p-1}{2}}$ and $d_0=1$. Then $d=\prod_{y\in C_2^n} d_y^{|\im f|/2}$ because the exponent of every prime dividing $d$ must be $|\im f|$ divided by the ramification degree, which is $2$. The field $\Q(\sqrt{a},\sqrt{b})$ for $a,b\in \Z$ odd and squarefree corresponds to the factorization $d= (ab)^2  = a^2 \cdot b^2 = d_{v_1}\cdot d_{v_2}$ for $v_1, v_2$ a basis of $C_2^2$. This gives a one-to-one correspondance between maps $f:G_\Q\rightarrow C_2^n$ of discriminant and factorizations $d=\prod_{y\in C_2^n} d_y^{|\im f|/2}$, since any such factoriztion defines a map by sending $\tau_p\mapsto y$ if $p\mid d_y$. In this case, conditions 1 and 2 on the factorizations are trivially satisfied because $|y|=2$ and $d$ is odd.

For this lemma, we are restricting the possible images of $f(\tau_p)$ to a specific subset of elements $Y_E\subset G$ and giving a correspondence between these maps and certain factorizations of the discriminant. This subset is specifically chosen so that we always have $\Res^{G}_{f(I_p)}([E])=0$, one of the conditions that we showed in the previous section was necessary to solve the Brauer embedding problem. Condition 3 is then equivalent to forcing the generator $\tau_\infty\in I_\infty$ to satisfy $f(\tau_\infty)\in Y_E$.
\end{example}

\begin{proof}[Proof of \ref{lemma:factorization}]
Define $d_y:=\prod_{f(\tau_p)=y} (p^*)^{a_p}$ a discriminant where $p^{a_p[\im f:\langle y\rangle]}\mid\mid \disc(f)$ exactly divides $\disc(f)$ for $y\ne 0$, and $d_0=1$. These are all clearly coprime. There exists a continuous homomorphism $f_y:G_\Q^{ab}=\prod I_p\rightarrow \langle y\rangle$ which equals $f$ on $I_p$ for primes with $f(\tau_p)=y$ and $0$ on all other primes. Let $K_y=\overline{\Q}^{\ker f_y}$. Then $\disc(K_y)=d_y$ by the higher ramification formula for discriminants and $[K_y:\Q]=|y|$. It follows that $\prod K_y/K$ is an unramified extension, where the power of $p$ dividing $\disc(\prod K_y)$ is exactly $a_p[\prod K_y: K_y]=a_p(\prod |y'| )/|y|=a_p\prod_{y'\ne y}|y'|$. Thus we have
\begin{align*}
\prod_{y\in Y_E} (d_y)^{\prod_{y'\ne y}|y'|} &= \disc(\prod K_y)\\
&=\disc(K)^{[\prod K_y:K]}\\
&=\disc(f)^{\frac{1}{|f(G_\Q)|}\prod |y|}
\end{align*}
Raising both sides to the power $\frac{|f(G_\Q)|}{\prod |y|}$ yields the result up to a sign. The conditions on $d_y$ and $y$ follow immediately from $f(\tau_{2,0}\tau_{2,1})$ having order dividing $2$, and $|f(\tau_p)|\mid (p-1)p^\infty$.

Given such a factorization, define $f(\tau_p)=y\in Y_E$ where $p\mid d_y$. For the finite primes, this gives a well-defined map $f:G_\Q^{ab}=\prod I_p\rightarrow G$ by $|y|\mid (p-1)p^\infty$.  The other conditions on the factorization imply that this map is well-defined on $I_2$. $f(\tau_p)\in Y_E$ implies $\Res^G_{f(I_p)}([E])=0$ at all finite places.

We have now proven the theorem up to a sign, except for part 3 and the infinite place. We will first show that part 3 is equivalent to $\Res^{G}_{f(I_\infty)}([E])=0$, then show that signs of the factorization also match.

Given such a factorization, the signs are not an obstruction to constructing a map as $G_\Q^{ab}=\prod I_p$ a product over finite primes and we can get a map just coming from the finite primes present in the factorization. $\Res^G_{f(I_\infty)}([E])=0$ if and only if
\begin{align*}
f(\tau_{2,0}\tau_{2,1})\prod_{p\equiv 3\mod 4}f(\tau_p)^{|f(\tau_p)|/2} \in Y_E
\end{align*}
This is the same thing as
\begin{align*}
\prod_{2\mid d_y<0}y\prod_{d_y<0}y^{|y|/2}\in Y_E
\end{align*}
since $d_y<0$ for $d_y$ odd if and only if an odd number of primes $p\equiv 3\mod 4$ divide it with odd exponent and $y$ has even order. For $d$ even, we must additionally multiply by $f(\tau_{2,0}\tau_{2,1})$, a product of the at most two $y\in Y_E$ such that $d_y$ is even. Thus $d_y<0$ for $d_y$ even if and only if $f(\tau_{2,0}\tau_{2,1})=yy'\ne 1$ and $2\mid \#\{p\mid d_y: p\equiv 3\mod 4\}$ or $f(\tau_{2,0},\tau_{2,1})=1$ and $2\nmid \#\{p\mid d_y: p\equiv 3\mod 4\}$. This is the same thing as $yy' (y)^{|y|/2}$ being the contribution to the infinite inertia. So $\prod_{2\mid d_y<0}y\prod_{d_y<0}y^{|y|/2}\in Y_E$ is equivalent to $\Res^G_{f(I_\infty)}([E])=0$.

As for the signs of discriminants, consider the quotient map $h:\Gal(\prod K_y/\Q)=\prod\langle y\rangle\mapsto \Gal(\prod K_y/\Q)/2\Gal(\prod K_y/\Q)$, then
\begin{align*}
h\left((y^{|y|/2})_{d_y<0}\right)&=(h(y)^{|y|/2})_{d_y<0,2\nmid|y|/2}\\
&=(h(y)^{|y|/2})_{d_y<0,2\nmid [f(G_\Q):\langle y\rangle]} 
\end{align*}
is the generator of $I_\infty$, noting that $h(y)$ is nontrivial if and only if $2\nmid [f(G_\Q):\langle y\rangle]$ (if $2\mid d_y<0$ we must use $y^{|y|/2 +1}$). It suffices to show that $\{y\in Y_E: 2\nmid [f(G_\Q):\langle y\rangle]\}\le f(G_\Q)/2f(G_\Q)$ is linearly independent in the $\F_2$-vector space, because then there is a map sending all such $y\mapsto -1$ which factors through $h$, so that the sign of $\disc(f)$ is $-1$ to the power of $\sum_{\{y:d_y<0,2\nmid [f(G_\Q):\langle y\rangle]\}}\frac{|y||f(G_\Q)|}{4}$.

First, if $\{y\in Y_E: 2\nmid [f(G_\Q):\langle y\rangle]\}\ne \emptyset$, then the $2$-Sylow subgroup of $f(G_\Q)$ is cyclic. In particular, an element $x$ is a square if and only if $[f(G_\Q):\langle x\rangle]$ is even.

Suppose $\prod_{i=1}^n y_i^{a_i}$ is a square ( i.e. an element of $2f(G_\Q)$) with $[f(G_\Q):\langle y_i\rangle]$ odd for all $y_i\in \{y\in Y_E: 2\nmid [f(G_\Q):\langle y\rangle]\}$. Then
\begin{align*}
\left\langle \prod_{i=1}^{n-1}y_i^{a_i},y_n^{a_n}\right\rangle &\cong \left\langle \prod_{i=1}^{n-1}y_i^{a_i}\right\rangle\times \left\langle y_n^{a_n}\right\rangle \Bigg{/} \left\langle \prod_{i=1}^{n}y_i^{a_i}\right\rangle
\end{align*}
as subgroups of $f(G_\Q)/2f(G_\Q)$. We induct on $n$. Clearly if $n=1$, then $y_1^{a_1}\in 2f(G_\Q)$ if and only if $2\mid a_1$. Suppose this is true for $n-1$. Then if $\prod_{i=1}^n y_i^{a_i}$ is a square, it follows that if $2^b\mid\mid |f(G_\Q)|$ then $2^b \mid \mid |y_n|$. Suppose $a_n$ is odd, then $2^b\mid\mid \left|\left\langle \prod_{i=1}^{n-1}y_i^{a_i},y_n^{a_n}\right\rangle\right|$. Being a square implies $\left\langle \prod_{i=1}^{n}y_i^{a_i}\right\rangle=\langle x^2\rangle$, and so $2^b\nmid |x^2|$. By using the above isomorphism, we must have that $2^b\nmid \left|\left\langle \prod_{i=1}^{n-1}y_i^{a_i}\right\rangle\right|$, which implies it is a square. By induction, $a_i\equiv 0\mod 2$ for all $i=1,...,n-1$. Then we are back to the base case $y_n^{a_n}$.

Now, the sign of $d_y$ is $1$ if it is unramified at $\infty$, and $(-1)^{|y|/2}$ if ramified. This means that the sign of $\prod_{y\in Y_E} (d_y)^{[f(G_\Q):\langle y\rangle]}$ is $-1$ to the power of
\[
\#\{y: d_y<0\text{ and }[f(G_\Q):\langle y\rangle]\text{ odd}\}=\sum_{2\nmid [\im f:\langle y\rangle]}|y|/2
\]
If $2\mid \mid |f(G_\Q)|$ then $|y|/2=1=|y||f(G_\Q)|/4\mod 2$ and we get equality. If $4\mid |f(G_\Q)|$, then $|y|/2=0=|y||f(G_\Q)|/4\mod 2$ and we also get equality. This concludes the proof.
\end{proof}

The field $\prod_{y\in Y_E} K_y$ with Galois group $\prod_{y\in Y_E} \langle y\rangle$ over $\Q$ is going to be very important. Call this field $L$ with Galois group $H$ over $\Q$. The following lemma follows immediately from properties of the quotient map $H\rightarrow G$ sending $y\mapsto y$ for each $y\in Y_E$:

\begin{lemma}
$[,]_E:G\times G\rightarrow A$ lifts to a map $[,]_E:H\times H\rightarrow A$ by applying $[,]_E$ coordinatewise and adding up the results, i.e. satisfies the following commutative diagram
\[
\begin{tikzcd}
H\times H\dar \drar{[,]_E} \\
G\times G \rar{[,]_E} &A
\end{tikzcd}
\]
In particular, if $g:G_Q\rightarrow \prod_{y\in Y_E} \langle y\rangle$ is defined by $\tau_p\mapsto f(\tau_p)\in Y_E$, then $[f(\Frob_p),f(\tau_p)]_E=[g(\Frob_p),g(\tau_p)]_E$.
\end{lemma}

Thus it suffices to check the commutator condition $[f(\Frob_p),f(I_p)]_E=0$ in $\prod_{y\in Y_E}K_y$.

\begin{lemma}
Let $f:G_\Q\rightarrow G$ be the continuous homomorphism of discriminant $\disc(f)=\prod p^{e_p}$ corresponding to a factorization $\disc(f)=\prod_{y\in Y_E} (d_y)^{[\im f:\langle y\rangle]}$ and let $g:G_Q^{ab}\rightarrow H$ be the map sending $\tau_p\mapsto f(\tau_p)\in Y_E$. Then
\begin{align*}
[g(\Frob_p),g(\tau_p)]_E &= \sum_{y\in Y_E} \sum_{q\mid d_y}\Leg{p}{q^{b_q}}_{\exp(A)}[y,g(\tau_p)]_E
\end{align*}
where $\Leg{p}{q^{b_q}}_{n}\in \Z/n\Z$ is the image of $p$ under the quotient map 
\[
\Z_q^\times\rightarrow \Z_q^\times/(\Z_q^\times)^{(n,q^{b_q-1}(q-1))}\hookrightarrow\Z/n\Z.
\]
and $b_q=e_q/[f(G_\Q):\langle y\rangle]$
\end{lemma}

\begin{proof}
We have that the Frobenius element $\Frob_p\in G_\Q^{ab}=\prod I_q$ is $(p)_q$ the equivalence class of $p$ in each coordinate. Therefore, under $g$ we get that
\begin{align*}
g(\Frob_p)&=g\left(\sum_{q\mid \disc(f)} (p).\tau_q \right)\\
&=\sum_{y\in Y_E} g\left(\sum_{q\mid d_y} (p).\tau_q \right)\\
&=\sum_{y\in Y_E}\sum_{q\mid d_y} (p).y
\end{align*}
The equivalence class of $p$, denoted above by $(p)$ acts on $\langle y\rangle$ by multiplication as an element of $\Z/|y|\Z$, namely $\Leg{p}{q^{e_q/[f(G_\Q):\langle y\rangle]}}_{|y|}$ for each $q\mid d_y$. The proof then follows from the bilinearity of $[,]_E$.
\end{proof}

\begin{definition}
We define $\Leg{b}{a}_n:\Z^\times\rightarrow \Z/n\Z$ by $\Leg{b}{a}_n=\sum_{q^e_q\mid\mid a}\Leg{\cdot}{q^{{e_q}}}_n$ for any $(a,b)=1$ where $a$ has prime factorization $\prod q^{e_q}$.
\end{definition}

We can now prove the main theorem for generalized Lemmermeyer factorizations as an immediate consequence of the work in this section, which we restate here for convenience:

\main*

\begin{proof}
Part $1$ is equivalent to $\Res^G_{f(I_p)}([E])=0$ by Lemma \ref{lemma:factorization}.

For part 2, we have that $q\mid d_{y'}$ if and only if $y'=g(\tau_q)$. So we then have
\begin{align*}
\sum_{y\in Y_E} \Leg{p}{d_{y}}_{\exp(A)} [y,f(\tau_p)]_E &=\sum_{y \in Y_E}\sum_{q\mid d_{y}} \Leg{p}{q^{e_q}}_{exp(A)} [y,g(\tau_p)]_E\\
&=[g(\Frob_p),g(I_p)]_E\\
&=[f(\Frob_p),f(I_p)]_E
\end{align*}
Which is $0$ if and only if $[f(\Frob_p),f(I_p)]_E=0$.
\end{proof}

\section{Certain Unramified Metabelian Extensions}\label{meta}

Lemmermeyer used factorizations of the kind found in this paper to classify unramified quaternion $H_8$-extensions of quadratic number fields \cite{lemmermeyer1}. In this section, we will detail how to use the results of the previous section to classify certain unramified metabelian extensions of abelian number fields in similar situations.

When discussing unramified $G$-extensions of number fields $L/K$, it is useful to consider $\Gal(L/\Q)$ as was done in \cite{alberts1}\cite{bhargava1}\cite{wood1}.

\begin{definition}
Call a pair $(G,G')$ with $G\le G'$ of finite index \textbf{admissible} if
\begin{enumerate}
\item{$G\normal G'$}
\item{$G'$ is generated by the set $Y_{G'}\cap (G' - G)$}
\end{enumerate}
\end{definition}

Given a $G'/G$-extension $K/\Q$, we call $L/K$ a $(G,G')$-extension if $L/\Q$ is normal, $\Gal(L/K)=G$, and $\Gal(L/\Q)=G'$. By properties of Galois theory, we know that if there exists an unramified $(G,G')$-extension $L/K$, then $(G,G')$ must be admissible.

Lemmermeyer noted in \cite{lemmermeyer1} that there existed a unique admissible pair $(H_8,G')$ with $[G':H_8]=2$, and then classified unramified $(H_8,G')$-extensions. This classification took advantage of the following property:

\begin{definition}
Call a pair $(G,G')$ \textbf{central} if
\[
[G',G']\le Z(G')
\]
where $Z(G')$ is the center of the group $G'$.
\end{definition}

\begin{proposition}\label{prop:redundancy}
Let $(G,G')$ be a central admissible pair. Then $\Aut([G',G'])$ acts transitively on $\{[E]\in H^2((G')^{ab},[G',G']):E\cong G'\}$ with stabilizers isomorphic to $\Aut([G'])$
\end{proposition}

\begin{proof}
Given $[E]$ with an embedding $i:[G',G']\hookrightarrow E$, $\alpha\in\Aut([G',G'])$ acts on it by $i\mapsto i\circ\alpha$. Because $[G',G']=[E,E]$ is the commutator subgroup, every such isomorphism $[E]\cong [E]$ must arise from this action, showing that the stabilizer is a quotient of $\Aut([E])$. Moreover, every extension is completely determined by the embedding of $[G',G']$ as the commutator subgroup, because the quotient is just the abelianization map which is characteristic (i.e. invariant under any automorphism which is trivial on the abelianization). Therefore the the stabilizer is isomorphic to $\Aut([G'])$ and because $\Aut([G',G'])$ acts transitively on $\text{Isom}([G',G'],[E,E])$, it must act transitively on $\{[E]:E\cong G'\}$.
\end{proof}

This will be sufficient to prove the following two corollaries:

\begin{corollary}
\label{corollary:existence}
Let $(G,G')$ be a central admissible pair and $K/\Q$ an abelian extension with $\overline{f}:G_\Q\twoheadrightarrow\Gal(K/\Q)=G'/G$. There exists an unramified $(G,G')$-extension $L/K$ if and only if their exists a factorization $\disc(K)^{[G:[G',G']]}=\prod_{y\in Y_{G'}} (d_y)^{[(G')^{ab}:\langle y\rangle]}$ into coprime discriminants (except that at most two are allowed to be even) satisfying the following:
\begin{enumerate}
\item{
\begin{enumerate}
\item{$Y_{G'}=\{y\in (G')^{ab}: \Res^{(G')^{ab}}_{\langle y\rangle}([G'])=0\}$}

\item{$|y|\mid (p-1)p^\infty$ for every prime dividing $d_y$.}

\item{If $d_y,d_{y'}$ are even with $y\ne y'$ then $yy'$ has order dividing $2$. If $d_y$ is a unique even factor and $4\mid\mid d_y$, then $y$ has order dividing $2$.}

\item{$\prod_{d_y<0}y^{|y|/2}\in Y_{G'}$.}

\item{$p\mid d_y$ implies $yG=\overline{f}(\tau_p)\in G'/G$}

\item{$(G')^{ab}=\langle y : d_y\ne 1\rangle$}

\item{$d_y=1$ for all $y\in Y_{G'}\cap G$.}
\end{enumerate}
}

\item{
For every prime $p\mid d$
\begin{align*}
\sum_{y\in Y_{G'}}\Leg{p}{d_y}_{\exp([G',G'])}[y,f(\tau_p)]_{G'}= 0
\end{align*}
where $\Leg{p}{d_y}_{\exp([G',G'])}\in\Z/\exp([G',G'])\Z$.
}
\end{enumerate}
\end{corollary}

\begin{corollary}
\label{corollary:count}
Let $(G,G')$ be a central admissible pair and $K/\Q$ an abelian extension with $\overline{f}:G_\Q\rightarrow \Gal(K/\Q)=G'/G$. Suppose there exists an unramified $(G,G')$-extension $L/K$ with $[\Gal(L/\Q)]=[G']\in H^2((G')^{ab},[G',G'])$ corresponding to a factorization $\disc(K)^{[G:[G',G']]}=\prod_{y\in Y_{G'}}(d_y)^{[(G')^{ab}:\langle y\rangle ]}$ as in Corollary \ref{corollary:existence}.  Then there are exactly
\[
\frac{\prod_{y\in Y_{G'}}(\#[G',G'][|y|])^{\omega(d_y)}}{\#\Aut([G'])\#\Hom(G',[G',G'])}
\]
such extensions, where $[G',G'][n] = \{g\in [G',G']: g^n=0\}$ is the $n$-torsion subgroup.
\end{corollary}

The idea is that these corollaries generalize \ref{theorem:C4} and \ref{theorem:H8}. The first one generalizes the factorization $d=d_1 d_2 d_3$ and Dirichlet conditions $\Leg{d_1 d_2}{p_3}=\Leg{d_1 d_3}{p_2}=\Leg{d_2 d_3}{p_1}=1$ from \ref{theorem:H8} and the second one generalizes the number $2^{\omega(d)-3}$ of such extension corresponding to a factorization in \ref{theorem:H8}. The conditions in these results are significantly denser in order to account for more cases. The reader may be interested in reading Section \ref{HeisenbergGroups} before the proofs of these corollaries, where we present a worked out example of the result of this section for the Heisenberg group $H(\ell^3)$, the unique nonabelian group of cardinality $\ell^3$ and exponent $\ell$.

\begin{proof}[Proof of \ref{corollary:existence}.]
Fix a choice of $[G']\in H^2((G')^{ab},[G',G'])$. We will prove the theorem for this choice, and then at the end of the proof show that the result for one choice implies the result for every such choice.

Because $(G,G')$ is central and admissible, we know that $G'$ is a central $[G',G']$-extension of $(G')^{ab}$. Setting $A=[G',G']$, $E=G'$ gives us the set-up of the previous sections. Lemma 3.1 shows that 1(a)(b)(c)(d)(e) is equivalent to the existence of a map $f:G_\Q\rightarrow (G')^{ab}$ which is a lift of the quotient map $G_\Q\rightarrow G'/G$ defining $K/\Q$. 1(f) is equivalent to this map being surjective, and so is equivalent to the existence of a $(G')^{ab}$-extension $M/\Q$ such that $M/K$ is unramfied. In Theorem \ref{thm:main}, we get that there exists an unramified lift of $f$ to $\widetilde{f}:G_\Q\rightarrow G'$ if and only if part 2 holds. Furthermore, $d_y=1$ implies that no ramification is sent to $y$, so it follows that $\widetilde{f}$ is unramified over the quotioent map $G_\Q\rightarrow \Gal(K/\Q)$ if and only if $d_y=1$ for all $y\in Y_{E}\cap G$, i.e. condition 1(g), holds.

It then suffices to check whether or not $\widetilde{f}$ is surjective, by Galois theory. We have $f$ is surjective by construction, so suppose that $\widetilde{f}$ is not surjective. Then $\widetilde{f}(G_\Q) \le G'$ is a subgroup which surjects onto $(G')^{ab}$. Therefore $G' = \widetilde{f}(G_\Q) A$ as $A=[G',G']$. If $A\le \widetilde{f}(G_\Q)$ then we are done. Let $B=A\cap \widetilde{f}(G_\Q)$. Then we have $G'/B \cong \widetilde{f}(G_\Q)/B\times A/B$ since $\widetilde{f}(G_\Q)/B\cap A/B=1$. However, $\widetilde{f}(G_\Q)/B\cong \widetilde{f}(G_\Q)/\widetilde{f}(G_\Q)\cap A\le G'/A$ is abelian, implying that $G'/B$ is abelian. Thus $[G',G']=A\le B=\im\widetilde{f}\cap A$, concluding the proof for fixed $[G']$.

Because $[G',G']$ is the commutator subgroup, the previous proposition tells us that different choices of $[G']$ only differ by embedding $[G',G']$ up to automorphism. This implies that $[,]_{G'}$ differs up to automorphism and $\Res^{(G')^{ab}}_{\langle y\rangle}([G'])$ is fixed up to isomorphism. This implies that $Y_{G'}$ is independent of the choice of $[G']$ and $[,]_{G'}$ is fixed up to automorphism of $[G',G']$, for which the equation in part 2 is invariant because automorphisms of $[G',G']$ are $\Z/\exp([G',G'])\Z$-linear maps.
\end{proof}

\begin{proof}[Proof of \ref{corollary:count}.]
Any unramifed $(G,G')$-extension $L/K$ corresponding to the factorization $\disc(K)^{[G:[G',G']]}=\prod_{y\in Y_{G'}}(d_y)^{[(G')^{ab}:\langle y\rangle]}$ contains the field $\widetilde{K}$ given by the continuous homomorphism $f:G_\Q^{ab}\rightarrow (G')^{ab}$ with $f(\tau_p)=y$ for $p\mid d_y$ by construction. So it suffices to count unramified $[G',G']$-extensions of $\widetilde{K}$ whose corresponding extension isomorphism class is $[G']$. By assumption $[G',G']$ is abelian and central, so any such extension must come from a surjective homomorphism $\Cl(\widetilde{K})\rightarrow [G',G']$. Let $H(\widetilde{K})$ be the Hilbert class field of $\widetilde{K}$. Then the exact sequence
\[
\begin{tikzcd}
1 \rar & \Cl(\widetilde{K}) \rar & \Gal(H(\widetilde{K})/\Q) \rar & \Gal(\widetilde{K}/\Q) \rar &1
\end{tikzcd}
\]
gives rise to the following inflation-restriction sequence with trivial action on $[G',G']$:
\[
\begin{tikzcd}
0 \rar & H^1(\Gal(\widetilde{K}/\Q),[G',G']) \rar & H^1(\Gal(H(\widetilde{K})/\Q),[G',G']) \dlar\\
{} &H^1(\Cl(\widetilde{K}),[G',G'])^{\Gal(\widetilde{K}/\Q)} \rar & H^2(\Gal(\widetilde{K}/\Q),[G',G'])
\end{tikzcd}
\]
where the last map sends $f:\Cl(\widetilde{K})\rightarrow [G',G']$ to $[\Gal(H(\widetilde{K})^{\ker f}/\Q)]$. By assumption $[G']$ is in the image of this map. So it follows that there are
\begin{align*}
H^1(\Gal(H(\widetilde{K})/\Q),[G',G']) / H^1(\Gal(\widetilde{K}/\Q),[G',G'])
\end{align*}
such homomorphisms $f:\Cl(\widetilde{K})\rightarrow [G',G']$ giving rise to $[G']$. Notice that because $[G',G']$ is abelian, any homomorphism $\Gal(H(\widetilde{K})/\Q)\rightarrow [G',G']$ must factor through $\Gal(\widetilde{K}^{gen}/\Q)$, the Galois group of the genus field. If $f_p:G_\Q^{ab}\rightarrow \langle y\rangle$ is the map sending $\tau_p\mapsto y$ and $\tau_q\mapsto 0$ for any prime $q\ne p$ then we can describe the genus field as $\prod_{p\mid \disc(K)} \overline{\Q}^{\ker f_p}$ with Galois group $\prod_{p\mid \disc(K)} \langle f(\tau_p)\rangle$ over $\Q$. Thus
\begin{align*}
H^1(\Gal(H(\widetilde{K})/\Q),[G',G']) &= \Hom\left(\prod_{p\mid \disc(K)}\langle f(\tau_p)\rangle, [G',G']\right)\\
&=\prod_{p\mid \disc(K)} \Hom(\langle f(\tau_p)\rangle,[G',G'])\\
&=\prod_{p\mid \disc(K)} [G',G'][|f(\tau_p)|]\\
&=\prod_{y\in Y_{G'}}\prod_{p\mid d_y} [G',G'][|y|]
\end{align*}
where $G[n]$ denotes the $n$-torsion of an abelian group. So it follows that
\[
\#H^1(\Gal(H(\widetilde{K})/\Q,[G',G']) = \prod_{y\in Y_{G'}} (\#[G',G'][|y|])^{\omega(d_y)}
\]
Then the number of such homomorphisms $\Cl(\widetilde{K})\rightarrow [G',G']$ giving rise to $[G']$ is
\begin{align*}
\frac{\prod_{y\in Y_{G'}}\#[G',G'][|y|]^{\omega(d_y)}}{\#H^1(\Gal(\widetilde{K}/\Q),[G',G'])} &= \frac{\prod_{y\in Y_{G'}}\#[G',G'][|y|]^{\omega(d_y)}}{\#\Hom((G')^{ab},[G',G'])}
\end{align*}
Since every homomorphism $G'\rightarrow [G',G']$ must factor through $(G')^{ab}$, we can replace the bottom with $\Hom(G',[G',G'])$.

Normally, there are more homomorphisms than corresponding fields, so one might expect that we need to divide by $\Aut([G',G'])$ to eliminate redundancy. However, the redundancy coming from automorphisms of $[G',G']$ in the case can produce different isomorphism classes of $[G']$. The Proposition \ref{prop:redundancy} shows that the subgroup $\Aut([G'])\hookrightarrow \Aut([G',G'])$ parametrizes all the redundancy, concluding the proof.
\end{proof}

\section{Heisenberg Groups}\label{HeisenbergGroups}
Fix an odd prime $\ell$ and consider the Heisenberg group
\[
H(\ell^3) = \langle x,y: [x,y]=z, x^\ell=y^\ell=z^\ell=1, [x,z]=[y,z]=1\rangle
\]
This is the unique nonabelian group of order $\ell^3$ and exponent $\ell$. We define
\[
E = H(\ell^3) \rtimes C_\ell = \langle x,y, \sigma: [x,y]=[x,\sigma]=[y,\sigma]=z, x^\ell=y^\ell=z^\ell=\sigma^\ell=1, [x,z]=[y,z]=1\rangle
\]
In particular, $\langle z\rangle=[E,E] \le Z(E)$ giving us a central exact sequence
\[
\begin{tikzcd}
1\rar &C_\ell \rar & E \rar &C_\ell^3 \rar &1
\end{tikzcd}
\]
 This makes $(H(\ell^3),E)$ a central admissible pair. Moreover, because every element of $E$ has order $\ell$, it follows that $Y_E=E^{ab}=C_\ell^3$. We will prove the following classification as an example of Corollary \ref{corollary:existence} and Corollary \ref{corollary:count}:

\begin{theorem}
Let $K/\Q$ be a cyclic degree $\ell$ extension of discriminant $d=p^{\ell-1}q^{\ell-1}r^{\ell-1}$ coprime to $\ell$. Fix $g\in \Gal(K/\Q)$ as a generator of every nontrivial inertia group. Then there exists an unramified $(H(\ell^3),E)$ extension $L/K$ if and only if there exists $A,B,C\in\Z/\ell\Z$ with $A+B+C\ne 0$ and
\begin{align*}
\Leg{p}{q^{\ell-1}}_\ell (-A)+\Leg{p}{r^{\ell-1}}_\ell B&=0\\
\Leg{q}{p^{\ell-1}}_\ell A+\Leg{q}{r^{\ell-1}}_\ell (-C)&=0\\
\Leg{r}{p^{\ell-1}}_\ell (-B)+\Leg{r}{q^{\ell-1}}_\ell C&=0
\end{align*}
Moreover, in that case there are exactly $\ell-1$ such extensions for each factorization, which all have distinct extension classes $[\Gal(L/\Q)]\in H^2(C_\ell^3,C_\ell)$.
\end{theorem}

\begin{proof}
Fix $\pi:E^{ab}\rightarrow C_\ell$ the quotient map given above. Without loss of generality, we choose generators $\tau_p,\tau_q,\tau_r\in G_\Q^{ab}$ such that $\overline{f}:G_\Q^{ab}\rightarrow \Gal(K/\Q)=C_\ell$ satisfies $\overline{f}(\tau_p)=\overline{f}(\tau_q)=\overline{f}(\tau_r)=g$.

We will walk through the conditions of the corollary for this case, describing what they each correspond to. Suppose we have a factorization
\[
\disc(K)^{[H(\ell^3):[E,E]]}=d^{\ell^2}=\prod_{y\in Y_{E}} (d_y)^{[E^{ab},\langle y\rangle]}=\prod_{y\in Y_E} d_y^{\ell^2}
\]
in other words, $d=\prod_{y\in Y_E} d_y$. This corresponds to a lift
\[
\begin{tikzcd}
{} & G_\Q\dar{\overline{f}}\dlar[swap]{f}\\
E^{ab}\rar & C_\ell
\end{tikzcd}
\]
as in Lemma \ref{lemma:factorization} if and only if the factorization satisfies 1(a)(b)(c)(d)(e) of Corollary \ref{corollary:existence}.

We will now walk through each part of Corollary \ref{corollary:existence}, describing what properties they imply and verifying them for our pair $(H(\ell^3),E)$:
\begin{enumerate}
\item{\begin{enumerate}
\item{
$Y_E$ is the set of possible images of ramification in $E^{ab}$, which is equivalent to having $\Res^{E^{ab}}_{f(I_p)}([E])=0$ for fintie primes. We see that every element of $E$ has order $\ell$, which implies that $Y_E=E^{ab}=C_\ell^3$.
}

\item{
By the existence of the field $K$, we must have $\ell\mid p-1,q-1,r-1$. Every element of $Y_E$ has order $\ell$, so this part is vacuous.
}

\item{
$K$ cannot be ramified at $2$, so none of the factors are even making this part vacuous as well.
}

\item{
$Y_E=E^{ab}$ makes this part vacuous, as everything belongs to $Y_E$. This part shows that $\Res^{E^{ab}}_{f(I_\infty)}([E])=0$.
}

\item{
By assumption, $\overline{f}(\tau_p)=\overline{f}(\tau_q)=\overline{f}(\tau_r)=g$, so $d_y\ne 1$ requires $\pi(y)=g$. This is necessary for there to exist $f:G_\Q\rightarrow E^{ab}$ a well-defined lift of $\overline{f}:G_\Q\rightarrow G$ corresponding to our factorization. Call the corresponding $y$ values $f(\tau_p)=y_p$, $f(\tau_q)=y_q$, and $f(\tau_r)=y_r$ respectively.
}

\item{
This is equivalent to the lift $f:G_\Q\rightarrow E^{ab}$ being surjective, i.e. $E^{ab}=\langle y_p,y_q,y_r\rangle$. This is true if and only if $\ker \pi = \langle y_py_q^{-1}, y_py_r^{-1}\rangle$, i.e. $y_py_q^{-1}, y_py_r^{-1}$ are linearly independent by $\dim_{\F_\ell}\ker \pi=2$.
}

\item{
This is equivlent to $f$ being unramified over $\overline{f}$. Given a choice of $y_p,y_q,y_r$ generating $E^{ab}$ satisfying the above conditions, this is vacuoucly true as $\pi(y_x)=g\ne 1$, so $y_x\not\in H(\ell^3)$.
}
\end{enumerate}
}
\end{enumerate}
This implies that part 1 of Corollary \ref{corollary:existence} is satisfied if and only if the factorization is given by $d=\prod_{\pi(y)=g}d_y$ and $\{y_py_q^{-1},y_py_r^{-1}\}$ are linearly independent. Given these, part 1 gives us a lift $f:G_\Q\rightarrow E^{ab}$ with $\Res^{E^{ab}}_{f(I_p)}([E])=0$ for all primes.

\begin{enumerate}
\item[2.]{
This is equivalent to $[f(\Frob_p),f(I_p)]_E=0$, and can be distilled into three equations:
\begin{align*}
\Leg{p}{q^{\ell-1}}_\ell[y_q,y_p]_E+\Leg{p}{r^{\ell-1}}_\ell[y_r,y_p]_E&=0\\
\Leg{q}{p^{\ell-1}}_\ell[y_p,y_q]_E+\Leg{q}{r^{\ell-1}}_\ell[y_r,y_q]_E&=0\\
\Leg{r}{p^{\ell-1}}_\ell[y_p,y_r]_E+\Leg{r}{q^{\ell-1}}_\ell[y_q,y_r]_E&=0
\end{align*}
since the factors will be trivial whenever $d_y=1$.
}
\end{enumerate}

Corollary \ref{corollary:existence} tells us that there exists an unramified $(H(\ell^3),E)$-extension $L/K$ if and only if parts 1 and 2 hold.

Consider the presentation given in the previous section
\[
E=\langle x,y,\sigma: [x,y]=[x,\sigma]=[y,\sigma]=z,x^\ell=y^\ell=z^\ell=\sigma^\ell=1,[x,z]=[y,z]=1 \rangle
\]
and without loss of generality let $\sigma$ be a preimage of $g$. Then we have the bilinear map
\[
[,]_E:\langle \overline{x},\overline{y},\overline{\sigma}:\overline{x}^\ell=\overline{y}^\ell=\overline{\sigma}^\ell=1\rangle^{ab}\rightarrow \langle z\rangle = \Z/\ell\Z
\]
Where we write $\Z/\ell\Z$ additively. Is defined by it's values on the basis $\overline{x},\overline{y},\overline{\sigma}$, given by being antisymmetric with $[\overline{x},\overline{y}]_{E}=[\overline{x},\overline{\sigma}]_E=[\overline{y},\overline{\sigma}]_E=z$.

By construction, we must have $y_p=v_p\overline{\sigma}$, $y_q=v_q\overline{\sigma}$, and $y_r=v_r\overline{\sigma}$ for $v_i\in\langle \overline{x},\overline{y}\rangle$. We have that $[y_i,y_j]_E=[v_i,v_j]_E+[v_i,\overline{\sigma}]_E+[\overline{\sigma},v_j]_E$ for any choice of $i,j\in\{p,q,r\}$ by a simple computation. Consider
\begin{align*}
[\overline{x}^{a_1}\overline{y}^{b_1},\overline{x}^{a_2}\overline{y}^{b_2}]_E &= a_1b_2-a_2b_1 = \text{det}\begin{pmatrix}a_1 & b_1\\a_2 & b_2\end{pmatrix}\\
[\overline{x}^{a_1}\overline{y}^{b_1},\overline{\sigma}]_E&=a_1+b_1
\end{align*}
So it follows that
\begin{align*}
[\overline{x}^{a_1}\overline{y}^{b_1}\overline{\sigma},\overline{x}^{a_2}\overline{y}^{b_2}\overline{\sigma}]_E&=a_1b_2-a_2b_1+a_1+b_1-a_2-b_2\\
&=(a_1-1)(b_2+1)-(a_2+1)(b_1-1)
\end{align*}
is also a determinant. The requirement that $y_py_q^{-1}, y_py_r^{-1}$ are linearly independent is the same as the requirement that $v_pv_q^{-1},v_pv_r^{-1}$ are linearly independent. A pair of vectors being linearly independent is equivalent to the matrix with those vectors as column vectors having nonzero determinant. So condition 1 is equivalent to $y_p,y_q,y_r$ mapping to $g$ under $\pi$ and satisfying
\begin{align*}
0 &\ne [v_pv_q^{-1},v_pv_r^{-1}]_E\\
&=-[v_p,v_r]_E-[v_q,v_p]_E+[v_q,v_r]_E\\
&=[v_r,v_p]_E+[v_p,v_q]_E+[v_q,v_r]_E\\
&=[y_r,y_p]_E+[y_p,y_q]_E+[y_q,y_r]_E
\end{align*}
Fix $A,B,C\in\Z/\ell\Z$ with $A+B+C\ne 0$. Suppose $y_p,y_q$ are chosen with $[y_p,y_q]_E=A$. Then we have that choosing $y_r=\overline{x}^{a_r}\overline{y}^{b_r}\overline{\sigma}$ with $[y_r,y_p]_E=B$ and $[y_q,y_r]_E=C$ is equivalent to solving a system of two linear equations in two variables. Write $y_p=\overline{x}^{a_p}\overline{y}^{b_p}\overline{\sigma}$ and $y_q=\overline{x}^{a_q}\overline{y}^{b_q}\overline{\sigma}$, then we want a solution to
\begin{align*}
(a_r-1)(b_p+1)-(a_p+1)(b_r-1)&=B\\
(a_q-1)(b_r+1)-(a_r+1)(b_q-1)&=C
\end{align*}
In other words
\begin{align*}
\begin{pmatrix}b_p+1 & -(a_p+1)\\-(b_q-1) & a_q-1\end{pmatrix}
\begin{pmatrix}a_r \\ b_r\end{pmatrix} &= \begin{pmatrix}B-b_p+a_p\\C+a_q-b_q\end{pmatrix}
\end{align*}
where the matrix has determinant
\[
(b_p+1)(a_q-1)-(a_p+1)(b_q-1) = [y_q,y_p]_E = -A
\]
WLOG by symmetry we may assume that $A\ne 0$ by $A+B+C\ne 0$, therefore we may find such a choice for $y_r$. This shows that any such choice of $A,B,C$ can be realized by $y_p,y_q,y_r$, so the equations on Dirichlet characters in this theorem are equivalent to condition 2 of Corollary \ref{corollary:existence} for some choice of $y_p,y_q,y_r$ satisfying condition 1.

As for the number of such extensions corresponding to a factorization, notice that plugging $[E,E]=C_\ell$, $|y|=\ell$, and $E^{ab}=C_\ell^3$ into the formula in Corollary \ref{corollary:count} tells us that there are
\[
\frac{\ell^{\omega(d)-3}}{\#\Aut([E])} = \frac{1}{\Aut([E])}
\]
such extensions for the choice of $[E]$. This implies that $\Aut([E])=1$, which can be checked group theoretically. There are $\#\Aut([E,E])/\#\Aut([E])=(\ell-1)/1$ isomorphism classes of $[E]$ by Proposition \ref{prop:redundancy}, concluding the proof.
\end{proof}

\section{Acknowledgements}

I would like to thank my advisor Nigel Boston for many helpful conversations, advice, and recommendations. I would also like to thank Yuan Liu for helpful conversations about the Bruaer embedding problem. This work was done with the support of National Science Foundation grant DMS-1502553.

\bibliographystyle{abbrv}

\bibliography{BAreferences}

\begin{thebibliography}{10}

\bibitem{alberts1}
B.~Alberts.
\newblock {Cohen-Lenstra} moments for some nonabelian groups, Aug 2016.
\newblock Preprint available at \url{https://arxiv.org/abs/1606.07867}.

\bibitem{alberts-klys1}
B.~Alberts and J.~Klys.
\newblock The distribution of {$H_{8}$}-extensions of quadratic fields, Jun
  2017.
\newblock Preprint available at \url{https://arxiv.org/abs/1611.05595}.

\bibitem{bhargava1}
M.~Bhargava.
\newblock The geometric sieve and the density of squarefree values of invariant
  polynomials, Jan 2014.
\newblock Preprint available at
  \url{https://arxiv.org/abs/1402.0031?context=math}.

\bibitem{cohen-lenstra1}
H.~Cohen and H.~W. Lenstra.
\newblock Heuristics on class groups of number fields.
\newblock {\em Lecture Notes in Mathematics Number Theory Noordwijkerhout
  1983}, page 33–62, 1984.

\bibitem{fouvry-kluners1}
E.~Fouvry and J.~Kl{\"u}ners.
\newblock On the 4-rank of class groups of quadratic number fields.
\newblock {\em Inventiones Mathematicae}, 167(3):455–513, 2006.

\bibitem{fueter1}
R.~Fueter.
\newblock Der {Klassenk{\"o}rper} der quadratischen {K{\"o}rper} und die
  komplexe {Multiplikation}.
\newblock {\em Diss. G{\"o}ttingen}, 1903.

\bibitem{gerth1}
F.~Gerth.
\newblock The 4-class ranks of quadratic fields.
\newblock {\em Inventiones Mathematicae}, 77(3):489–515, 1984.

\bibitem{gras1}
G.~Gras.
\newblock Sur les $\ell$-classes d'id{\'e}aux dans les extensions cycliques
  relatives de degr{\'e} premier $\ell$.
\newblock {\em Ann. Inst. Fourier}, 23.3, 23.4:1--48, 1--44, 1973.

\bibitem{inaba1}
E.~Inaba.
\newblock {\"U}ber die {Struktur} der $\ell$-{Kassengruppe} zyklischer
  {Zahlk{\"o}rper} von {Primzahlgrad} $\ell$.
\newblock {\em J. Fac. Sci. Tokyo I}, 4:61--115, 1940.

\bibitem{klys1}
J.~Klys.
\newblock The distribution of p-torsion in degree p cyclic fields, Oct 2016.
\newblock Preprint available at \url{https://arxiv.org/abs/1610.00226v1}.

\bibitem{lemmermeyer1}
F.~Lemmermeyer.
\newblock Unramified quaternion extensions of quadratic number fields.
\newblock {\em Journal de Th{\'e}orie des Nombres de Bordeaux}, 9(1):51–68,
  1997.

\bibitem{lemmermeyer2}
F.~Lemmermeyer.
\newblock Class field towers, 2010.
\newblock available at
  \url{http://www.rzuser.uni-heidelberg.de/~hb3/publ/pcft.pdf}.

\bibitem{malle-matzat1}
G.~Malle and B.~H. Matzat.
\newblock {\em Inverse Galois theory}.
\newblock 1994.

\bibitem{martinet1}
J.~Martinet.
\newblock A propos de classes d'ideaux.
\newblock {\em S{\'e}minaire Th{\'e}orie des Nombres Bordeaux}, 5, 1971/1972.

\bibitem{nakagoshi1}
N.~Nakagoshi.
\newblock A construction of unramified abelian $\ell$-extensions of regular
  {Kummer} extensions.
\newblock {\em Acta Arith.}, 44:47--58, 1984.

\bibitem{nomura1}
A.~Nomura.
\newblock On the existence of unramified $p$-extensions.
\newblock {\em Osaka Journal of Mathematics}, 28:55--62, 1991.

\bibitem{nomura2}
A.~Nomura.
\newblock A note on unramified quaternion extensions over quadratic number
  fields.
\newblock {\em Proc. Japan Acad.}, 78:80--82, 2002.

\bibitem{nomura3}
A.~Nomura.
\newblock Notes on the existence of certain unramified $2$-extensions.
\newblock {\em Illinois Journal of Mathematics}, 46:1279--1286, 2002.

\bibitem{nomura4}
A.~Nomura.
\newblock Notes on the existence of unramified non-abelian $p$-extensions over
  cyclic fields.
\newblock {\em Proc. Japan Acad.}, 90:67--70, 2014.

\bibitem{odai1}
Y.~Odai.
\newblock On unramified cyclic extensiosn of degree $\ell$ of algebraic number
  fields of degree $\ell$.
\newblock {\em Nagoya Math. J.}, 107:135--146, 1987.

\bibitem{redei1}
L.~R{\'e}dei.
\newblock Die {Anzahl} der durch 4 teilbaren {Invariation} der {Klassengruppe}
  eines beliebigen quadratischen {Zahlk{\"o}rpers}.
\newblock {\em Math. Naturwiss. Anz. Ungar. Akad. d. Wiss.}, 10:338--363, 1932.

\bibitem{redei-reichardt1}
L.~R{\'e}dei and H.~Reichardt.
\newblock Die {Anzahl} der durch 4 teilbaren {Invarianten} der {Klassengruppe}
  eines beliebigen quadratischen {Zahlk{\"o}rpers}.
\newblock {\em J. Reine Angew. Math.}, 170:69--74, 1933.

\bibitem{reichardt1}
H.~Reichardt.
\newblock Zur {Struktur} der absoluten {Idealklassengruppe} im quadratischen
  {Zahlk{\"o}rper}.
\newblock {\em J. Riene Angew. Math.}, 170:75--82, 1933.

\bibitem{serre1}
J.-P. Serre.
\newblock {\em Topics in Galois Theory}.
\newblock A K Peters, 2008.

\bibitem{wood1}
M.~M. Wood.
\newblock Nonabelian {Cohen-Lenstra} moments, Feb 2017.
\newblock Preprint available at \url{https://arxiv.org/abs/1702.04644}.

\end{thebibliography}

\end{document}